\documentclass[11pt]{article}
\usepackage{amsmath}
\usepackage{amstext,amsbsy}
\usepackage{epsfig,multicol}
\usepackage{amsthm}
\usepackage{amssymb,latexsym}
\usepackage{amsfonts}
\usepackage{color}
\usepackage{tikz}
\usepackage{listings}
\usepackage{enumerate}
\usepackage{framed}
\usepackage{caption}
\usepackage{subcaption}
\usetikzlibrary{decorations.pathreplacing}

\theoremstyle{plain}
\newtheorem{theorem}{Theorem}[section]
\newtheorem{lemma}[theorem]{Lemma}
\newtheorem{corollary}[theorem]{Corollary}

\theoremstyle{definition}
\newtheorem{definition}[theorem]{Definition}

\newtheorem*{properties*}{Properties}

\newenvironment{definition*}[1][Definition]{\begin{trivlist}
\item[\hskip \labelsep {\bfseries #1}]}{\end{trivlist}}

\numberwithin{equation}{section}

\DeclareMathOperator{\cyc}{cyc}

\DeclareMathOperator{\Cyc}{Cyc}
\DeclareMathOperator{\Rmil}{Rl-min}
\DeclareMathOperator{\sor}{sor}
\DeclareMathOperator{\maj}{maj}
\DeclareMathOperator{\inv}{inv}

\DeclareMathOperator{\rmil}{rl-min}
\newcommand{\rmilb}{\rmil_{B}}

\DeclareMathOperator{\fmaj}{fmaj}
\DeclareMathOperator{\rmaj}{rmaj}

\DeclareMathOperator{\Bmap}{Bmap}

\DeclareMathOperator{\none}{n_1}
\DeclareMathOperator{\ntwo}{n_2}
\newcommand{\invb}{\inv_{B}}
\newcommand{\invd}{\inv_{D}}
\DeclareMathOperator{\Btmax}{Bt-max} 
\newcommand{\Btmaxb}{\Btmax_{B}}
\newcommand{\Btmaxd}{\Btmax_{D}}

\DeclareMathOperator{\Rlmin}{Rl-min}
\newcommand{\Rlminb}{\Rlmin_{B}}
\newcommand{\Rlmind}{\Rlmin_{D}}
\DeclareMathOperator{\sorb}{sor_B}
\newcommand{\Cycb}{\Cyc_{B}}
\DeclareMathOperator{\Cbtmax}{Cbt-max}

\DeclareMathOperator{\Des}{Des}
\DeclareMathOperator{\acode}{A-code}
\DeclareMathOperator{\bcode}{B-code}
\DeclareMathOperator{\mcode}{M-code}

\newcommand{\Desb}{\Des_{B}}
\newcommand{\majb}{\maj_{B}}
\DeclareMathOperator{\p}{p}
\newcommand{\Cbtmaxb}{\Cbtmax_{B}}

\DeclareMathOperator{\sgn}{sgn}

\begin{document}
\title{Sorting Index and Mahonian-Stirling Pairs for Labeled Forests}
\author{
Amy Grady and Svetlana Poznanovi\'c$^1$  \\ [6pt]
Department of Mathematical Sciences\\
Clemson University, Clemson, SC 29634  \\
}
\date{}
\maketitle
\begin{abstract} Bj\"orner and Wachs defined a major index for labeled plane forests and  showed that it has the same distribution as the number of inversions. We define and study the distributions of a few other natural statistics on labeled forests. Specifically, we introduce the notions of bottom-to-top maxima, cyclic bottom-to-top maxima, sorting index, and cycle minima.  Then we show that the pairs $(\inv, \Btmax)$, $(\sor, \Cyc)$, and $(\maj, \Cbtmax)$ are equidistributed. Our results extend the result of Bj\"orner and Wachs and generalize results for permutations. We also introduce analogous statistics for signed labeled forests and show equidistribution results which generalize results for signed permutations.

\end{abstract} 



{\renewcommand{\thefootnote}{} \footnote{\emph{E-mail addresses}:
agrady@clemson.edu (A. Grady), spoznan@clemson.edu (S.~Poznanovi\'c)}

\footnotetext[1]{The second author is partially supported by the NSF grant DMS-1312817.}


\section{Introduction} \label{S:introduction}

Let $F$ be a plane forest with a vertex set $V(F) = \{ v_1, \dots, v_n\}$.  We will draw $F$ with the roots on top and think of it as a Hasse diagram of the poset $(V(F), <_F)$. Throughout this paper, we will assume that the vertices  of $F$ are naturally indexed. That is, if $v_i <_F v_j$, then $i< j$.

A labeling $w$ of $F$ is a bijection
\[w: V(F) \rightarrow \{1, 2, \dots, n\}.\]  Let $\mathcal{W}(F)$ be the set of all labelings of a forest $F$.
For each vertex $v \in V(F)$, we will denote by $h_v$ the number of vertices of the subtree of $F$ rooted at $v$. In other words, $h_v$ is the size of the principal order ideal generated by $v$. The inversion number of a labeled forest $(F,w)$ is defined as 
\[ \mathrm{inv}(F,w) = \# \{(u,v) : u <_F v, w(u) > w(v)\}.\]
If the forest $F$ is a linear tree this is simply the inversion index of the corresponding permutation obtained by reading the labels of $F$ from bottom to top. The inversion index was generalized from permutations to trees by Mallows and Riordan~\cite{MR} who showed that the inversion enumerator for unordered labeled trees has  interesting properties. In this paper, we will be considering the inversion enumerator for all labelings of a fixed forest $F$. 

Bj\"orner and Wachs~\cite{BW1} extended another classical permutation statistic, the major index, to labeled forests. Namely, they defined the descent set of a labeled forest as 
\[ \Des(F,w) = \{v \in V(F) : w(v) > w(u), u \text{  is the parent of } v\},\]
the major index as
\[ \maj(F,w)= \sum_{v \in \Des(F,w)} h_v\]
and showed that the major index has the same distribution as the inversion index on labeled forests of fixed shape (see~\cite{LW} for a bijective proof): 
\[  \sum_{w\in\mathcal{W}(F)} q^{\maj(F,w)}=\sum_{w\in\mathcal{W}(F)} q^{\inv(F,w)}= \frac{n!}{\prod_{v \in V(F)}h_v} \prod_{v \in V(F)} [h_v] .\]
Here and throughout the text we  use $[n]$ to denote the $q$-integer $1+q+q^2+ \cdots +q^{n-1}$.

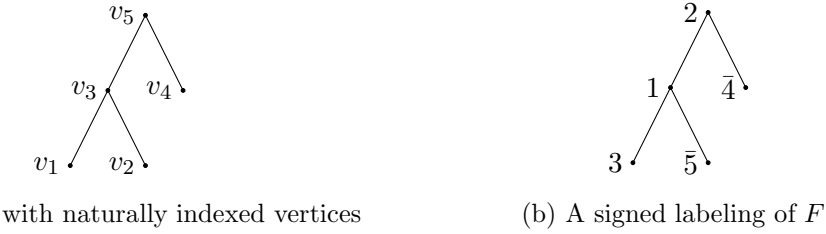
\begin{figure}[ht]
\centering
\begin{subfigure}[b]{0.45\textwidth}
\centering
\begin{tikzpicture} [scale = 0.5]
\draw (0,0) -- (1,2) -- (2,4) -- (3,2);
\draw (1,2) -- (2,0);
\draw [fill] (0,0) circle [radius=0.05];
\node [left] at (0,0) {$v_1$};
\draw [fill] (1,2) circle [radius=0.05];
\node [left] at (1,2) {$v_3$};
\draw [fill] (2,4) circle [radius=0.05];
\node [left] at (2,4) {$v_5$};
\draw [fill] (3,2) circle [radius=0.05];
\node [left] at (3,2) {$v_4$};
\draw [fill] (2,0) circle [radius=0.05];
\node [left] at (2,0) {$v_2$};
\end{tikzpicture}

\caption{A tree $F$ with naturally indexed vertices}
\label{fig:vertices}
\end{subfigure}
\begin{subfigure}[b]{0.45\textwidth}
\centering
\begin{tikzpicture} [scale = 0.5]
\draw (0,0) -- (1,2) -- (2,4) -- (3,2);
\draw (1,2) -- (2,0);
\draw [fill] (0,0) circle [radius=0.05];
\node [left] at (0,0) {3};
\draw [fill] (1,2) circle [radius=0.05];
\node [left] at (1,2) {1};
\draw [fill] (2,4) circle [radius=0.05];
\node [left] at (2,4) {2};
\draw [fill] (3,2) circle [radius=0.05];
\node [left] at (3,2) {$\bar{4}$};
\draw [fill] (2,0) circle [radius=0.05];
\node [left] at (2,0) {$\bar{5}$};

\end{tikzpicture}
\caption{A signed labeling of $F$}
\label{fig:almostmoon}
\end{subfigure}
\caption{A forest $F$ with a signed labeling}
\label{sgnl}
\end{figure}

A signed labeling of the forest $F$ of size $n$ is a one-to-one map
\[ w: V(F) \rightarrow \{\pm1, \pm2, \dots, \pm n\}\] such that  if $i \in w(V(F))$ then $-i \notin w(V(F))$ (see Figure~\ref{sgnl}). The set of all signed labelings of $F$ will be denoted by $\mathcal{W}_B(F)$. Chen et al.~\cite{CGG} extended the notion of inversions and major index to signed labeled forests, the latter one in two different ways. The inversion number $\inv_B$ for signed labelings is motivated by the length function for signed permutations, while the major indices $\fmaj$ and $\rmaj$ are based on the major indices of signed permutations introduced by Adin and Roichman~\cite{AR} and Reiner~\cite{Reiner}, respectively. The authors in~\cite{CGG} showed that
\[  \sum_{w \in\mathcal{W}_B(F)} q^{\fmaj(F,w)}= \sum_{w \in\mathcal{W}_B(F)} q^{\rmaj(F,w)}=\sum_{w \in\mathcal{W}_B(F)} q^{\inv_B(F,w)}= \frac{n!}{\prod_{v \in V(F)}h_v} \prod_{v \in V(F)} [2h_v] .\]

The inversion and major indices are so-called Mahonian statistics. Another permutation statistic from the same family is the sorting index~\cite{Petersen}. For a permutation $\sigma \in S_n$ there is a unique decomposition as a product of transpositions,  $\sigma = (i_1,j_1)(i_2,j_2) \cdots (i_k,j_k)$, such that $j_1<j_2<\cdots <j_k$ and $i_1<j_1, i_2<j_2,  \ldots , i_k<j_k$.  The sorting index is defined by $\sor(\sigma) = \sum_{r=1}^k (j_r-i_r)$.   The desired transposition decomposition can be found using the Straight Selection Sort algorithm. The algorithm first places $n$ in the $n$-th position by applying a transposition, then places $n-1$ in the $(n-1)$-st position by applying a transposition, etc. For example, consider the permutation $\sigma = 2413576$.  We have \[2413576 \overset{(67)} \rightarrow 2413567 \overset{(24)} \rightarrow 2314567 \overset{(23)} \rightarrow 2134567 \overset{(12)} \rightarrow 1234567.\] Therefore, $\sor(\sigma) = (2-1)+(3-2)+(4-2)+(7-6) = 5$. One of the goals of this paper is to generalize the notion of sorting index to labeled forests. 

In fact, we will be considering pairs of Mahonian and Stirling statistics, i.e, permutation statistics whose distribution is governed by the unsigned Stirling numbers of the first kind. Two well-known Stirling statistic are the number of cycles, $\cyc$, and the number of \textbf{r}ight-to-left \textbf{mi}nimum letters, $\rmil$. It is well known that
\[ \sum_{\sigma \in S_n} t^{\cyc(\sigma)} = \sum_{\sigma \in S_n} t^{\rmil(\sigma)} = \prod_{k=0}^{n-1} (t+k).\]


The motivation for this paper was  to extend some results for Mahonian-Stirling pairs on permutations to (signed) labeled forests. The pairs we consider are $(\inv, \Btmax)$, $(\maj, \Cbtmax)$, and $(\sor, \Cyc)$. We use capital letters to denote set-valued statistics. So, for example, $\Bmap$ is the set of bottom-to-top maximum positions. Precise definitions of the statistics $\Btmax$, $\Cbtmax$ (cyclic bottom-to-top maximum positions), $\sor$ (sorting index), and $\Cyc$ (minimal elements in cycles) will be given below. Our main result is that these three pairs of statistics are equidistributed over all (signed) labelings of a forest $F$. We give a bijective proof of this fact by mapping the labelings to certain integer sequences in three different ways. This also gives us an explicit formula for the generating function of each of the three pairs. Explicitly, we prove that
\begin{align*} \sum_{w \in \mathcal{W}(F)} q^{\inv(F,w)} \prod_{v \in \Btmax(F,w)} t_v = \sum_{w \in \mathcal{W}(F)} q^{\sor(F,w)} \prod_{v \in \Cyc(F,w)} t_v &= \sum_{w \in \mathcal{W}(F)} q^{\maj(F,w)} \prod_{v \in \Cbtmax(F,w)} t_v \\ &= \frac {n!}{\prod_{v \in V(F)} h_{v} } \prod _{v \in V(F)} \left([h_v] -1 +t_v\right), \end{align*} and
\begin{align*} \label{inv signed} \sum_{w \in \mathcal{W}_B(F)} q^{\invb(F,w)} \prod_{v \in \Btmaxb(F,w)} t_v &= \sum_{w \in \mathcal{W}_B(F)} q^{\sorb(F,w)} \prod_{v \in \Cycb(F,w)} t_v \\ &=  \frac {n!}{\prod_{v \in V(F)} h_{v} } \prod _{v \in V(F)} \left([2h_v] -1 +t_v\right). \end{align*}
When the forest is a linear tree, we show how these statistics specialize to known permutation statistics, and we discuss how our results are a generalization of some results for (signed) permutations.

The paper is organized as follows. Sections~\ref{S: inv},~\ref{S: sor}, and~\ref{S: maj}  deal with each of the pairs $(\inv, \Btmax)$, $(\maj, \Cbtmax)$, and $(\sor, \Cyc)$ separately. Whenever possible, we work with signed labeled forests and in our results we keep track of the negative signs in the labeling, so that the results for unsigned labeled forests are a corollary.  At places, we also discuss the case of even signed labeled forests, which is related to the case of even signed permutations. 


\section{Inversions and Bottom-to-Top Maxima}\label{S: inv}

Recall that a signed labeling of a forest $F$ is a one-to-one map $w: V(F) \rightarrow \{\pm1, \pm2, \dots, \pm n\}$ such that if $i \in w(V(F)) $ then $ -i \notin w(V(F))$.  As usual, we will denote $-i$ by $\bar{i}$.   A labeling is even-signed if the number of negative labels used is even. We will use $\mathcal{W}_B(F)$ and $\mathcal{W}_D(F)$ to denote the set of all signed labelings and the set of all even-signed labelings of a forest $F$, respectively. The type $B$ and type $D$ analogues of the inversion number of a labeled forest introduced by Bj\"orner and Wachs~\cite{BW1} was proposed by Chen et al.~\cite{CGG}. The definition is as follows. Let $\none(F,w)$ be the number of negative labels in $w$, and define \[\ntwo(F,w) = \#\{ (x,y) : x <_F y, w(x)+w(y)<0\}.\] Then the inversion number of a signed labeled forest is given by

\[\invb(F,w) =\inv(F,w) + \none(F,w) + \ntwo(F,w), \] while for $w \in W_D(F)$, the type $D$ inversion number is defined by
\[\invd(F,w)=\inv(F,w) + \ntwo(F,w).\]  For example, for the even signed labeled forest $(F,w)$ from Figure~\ref{sgnl}, $\invb(F,w) = 2 + 2+3 =7$ and $\invd(F,w) = 2 + 3 =5$. Note that if a signed labeling $w$ is in fact in $\mathcal{W}(F)$, then $\invb(F,w)=\inv(F,w)$. Chen et al.~\cite{CGG} showed that for a forest $F$ with $n$ vertices
\[\sum_{w\in\mathcal{W}_B(F)} p^{\none(F,w)} q^{\inv_B(F,w)}= \frac{n!}{\prod_{v \in V(F)}h_v} \prod_{v \in V(F)} (1+pq^{h_v})[h_v].\]
As a corollary, they obtained
\[\sum_{w\in\mathcal{W}_B(F)} q^{\inv_B(F,w)}= \frac{n!}{\prod_{v \in V(F)}h_v} \prod_{v \in V(F)} [2h_v]\]
and
\[\sum_{w\in\mathcal{W}_D(F)} q^{\inv_D(F,w)}= \frac{n!}{2\prod_{v \in V(F)}h_v} \prod_{v \in V(F)}(1+q^{h_v-1}) [h_v].\]
In this section we refine these results by looking at the joint distribution of inversions and bottom-to-top maxima.

\begin{definition} Let $F$ be a forest. For $w \in \mathcal{W}(F)$, we define the \emph{bottom-to-top maximum positions} to be
\[\Btmax(F,w) =  \{ v : w(v)>w(u) \text{ for all } u  <_F v\}.\]  For $w \in \mathcal{W}_B(F)$, we define the \emph{signed bottom-to-top maximum positions} to be
\[\Btmaxb(F,w) =  \{ v : w(v)>0 \text{ and } w(v)>|w(u)| \text{ for all } u  <_F v\}.\]  Finally, for $w \in \mathcal{W}_D(F)$, we define the even \emph{signed bottom-to-top maximum positions} to be
\[\Btmaxd(F,w) =  \{ v : v \text{ is not a leaf}, w(v)>0 \text{ and } w(v)>|w(u)| \text{ for all } u  <_F v\}.\]
\end{definition}

For example, for the signed labeled forest $(F,w)$ from Figure~\ref{sgnl}, $\Btmaxb(F,w) = \{v_1\}$ and $\Btmaxd(F,w) =\emptyset$. In this and following sections we will make use of maps between labeled forests and certain sequences that in the case of the symmetric group reduce to inversion tables. We will use $\mathrm{SE}_F$ and $\mathrm{SE}_F^B$ to denote the type A and type $B$ subexcedent sequences that correspond to a forest $F$, respectively: 
\[ \mathrm{SE}_F = \{ (a_1, \dots, a_n): a_i \in \mathbb{Z}, 0 \leq a_i \leq h_{v_i} -1 \},\]
\[ \mathrm{SE}_F^B = \{ (a_1, \dots, a_n): a_i \in \mathbb{Z}, 0 \leq a_i \leq 2h_{v_i} -1 \}.\]

For a forest $F$ and $w \in \mathcal{W}_B(F)$, we define its $\acode(F,w)$  to be the sequence $(a_1, \dots, a_n) \in \mathrm{SE}_F^B$ given by 
\[a_i = \# \{u :  u <_F v_i  \text{ and } w(u)>w(v_i)\} + \#\{u:  u <_F v_i  \text{ and } w(u)+w(v_i) <0 \} + \chi (w(v_i)<0),\] where $\chi$ is the truth indicator function.  
For example, the $\acode$ of $(F,w)$ from Figure~\ref{sgnl} is $(0,1,2,1,3)$.
\begin{lemma} \label{lem:B A-code info}
Let $w \in \mathcal{W}_B(F)$ and let $\acode(F,w)=(a_1,a_2,\dots,a_n)$. Then 
\begin{enumerate}
\item $\invb(F,w) = \sum _{i=1}^n a_i$ 
\item $\Btmaxb(F,w) = \{v_i : a_i = 0\}$
\item $w(v_i) < 0$ if and only if $h_{v_i}  \leq a_i \leq 2h_{v_i} -1$ and therefore $\none(F,w) = \#\{i: h_{v_i}  \leq a_i \leq 2h_{v_i} -1\}$.
\end{enumerate}
\end{lemma}

\begin{proof}
It is clear from the definition of the $\acode$ that $\invb(F,w) = \sum _{i=1}^n a_i$ because each $a_i$ is a sum of the amounts that the vertex $v_i$ contributes to $\inv(F,w)$, $\none(F,w)$, and $\ntwo(F,w)$.  Now, by definition, $v_i$ is a signed bottom-to-top  maximum position if and only if $w(v_i)>0$ and $w(v_i)>|w(u)|$ for all $u <_F v_i$.  Furthermore, $w(v_i)>|w(u)|$ for all $u<_F v_i$ if and only if $v_i$ does not create any inversions with vertices below it and $w(v_i)+w(u)>0$ for all $u <_F v_i$.  Therefore, $v_i$ is a signed bottom-to-top maximum position if and only if $a_i = 0$.  This proves the second part of the lemma. For the third part, note that if $w(v_i) > 0$ each vertex $u$ such that $u <_F v_i$ belongs to at most one of the sets  $\{v_j: v_j <_F v_i \text{ and } w(v_j) > w(v_i)\}$ and $\{v_j : v_j <_F v_i \text{ and } w(v_j)+w(v_i)<0\}$. Therefore, in this case $a_i < h_{v_i}$. On the other hand,  if $w(v_i) < 0$ each vertex $u$ such that $u <_F v_i$ belongs to at least one of these two sets and therefore $a_i \geq h_{v_i}$. 
\end{proof}

A labeling $w \in \mathcal{W}(F)$ is said to be natural if it preserves the order $<_F$. The $\acode$ is not a bijection between $\mathcal{W}_B(F)$ and $\mathrm{SE}_F^B$, but can be used to define the following bijection $\phi$ from $\mathcal{W}_B(F)$ to the set $\{ (w', (a_1, \dots,a_n)) :  w' \in \mathcal{W}(F) \text{ is a natural labeling  and } (a_1, \dots,a_n) \in \mathrm{SE}_F^B\}$. First we set $(a_1,\dots,a_n)$ to be the $\acode$ of $(F,w)$. The natural positive labeling $w'$ is obtained by a sequence of $n$ modifications applied to $w$ in the following way. Start with $w_n=w$. If the labeling $w_i$ has been defined for $i>0$, construct $w_{i-1}$  as follows.  Set $A_i = \{|w_i(u)| : u \leq_F v_i\}$.  Find the largest element in $A_i$, say it is $|w_i(v_j)|$, and define the new labeling $w_{i-1}$ of $F$ so that 
\begin{enumerate}
\item $w_{i-1}(v_i)=|w_i(v_j)|$ 
\item for all $u \nless_F v_i$, $w_{i-1}(u) = w_i(u)$ 
\item in $w_{i-1}$, the absolute values of the labels of the vertices below $v_i$ are given by $A_i \backslash \{|w_i(v_j)|\}$  so that  for all $u, u' <_F v_i$, $|w_{i-1}(u)| < |w_{i-1}(u')|$ if and only $|w_{i}(u)| < |w_{i}(u')|$ and  $\sgn w_{i-1}(u) = \sgn w_{i}(u)$ for all $u <_F v_i$. 
\end{enumerate}
Finally, we set $w'=w_0$.  

\begin{lemma}\label{phi}
If  the $\acode$ of $(F,w)$ is $(a_1, \dots, a_n)$, then the $\acode$ of $(F, w_k)$ defined in the previous paragraph is $(a_1, \dots, a_{k}, 0, \dots, 0)$, $1 \leq k \leq n$. Thus $w'$ has no inversions and is a natural positive labeling.
\end{lemma}
\begin{proof}
Assume that the $\acode$ of $(F, w_{i})$ is $(a_1, \dots, a_{i}, 0, \dots, 0)$. In both $w_i$ and $w_{i-1}$,  all vertices $u$ with $u \nless_F v_i$  form the same number of inversions with vertices below them. So, the corresponding entries in their A-codes are equal.  Furthermore, the choice of the label $w_{i-1}(v_i)$ is such that it's clear that the $i$-th entry of the $\acode$ of $(F, w_{i-1})$ is 0. What remains is to show that the entries of the two A-codes corresponding to the vertices $v_j$ below $v_i$ are the same. The third property of $w_{i-1}$ directly implies
$ \# \{u :  u <_F v_j  \text{ and } w_{i-1}(u)>w_{i-1}(v_j)\} =  \# \{u :  u <_F v_j  \text{ and } w_{i}(u)>w_i(v_j)\}$ and $\chi (w_{i-1}(v_j)<0) = \chi (w_i(v_j)<0)$. So, we only need to check that
 $$ \#\{u:  u <_F v_j  \text{ and } w_{i-1}(u)+w_{i-1}(v_j) <0 \}  = \#\{u:  u <_F v_j  \text{ and } w_i(u)+w_i(v_j) <0 \}.$$ Note that from $w_i$ to $w_{i-1}$, the labels below $v_i$ can stay the same, can increase by $1$ (in which case they were negative), or can decrease by $1$ (in which case they were positive). Therefore, $|(w_{i}(u)+w_{i}(v_j)) - (w_{i-1}(u)+w_{i-1}(v_j))|\leq 2$. Thus, a change in the sign of the sum of two labels can possibly occur only when $w_{i}(u)+w_{i}(v_j)\in \{-2, -1, 1, 2\}$. In case when $w_{i}(u)+w_{i}(v_j)\in \{-2, -1\}$ we have that one of $w_{i}(u), w_{i}(v_j)$ is positive while the other one is negative. So $w_{i-1}(u)+w_{i-1}(v_j) \leq w_{i}(u)+w_{i}(v_j)+1\leq 0$. Since all labels are different in absolute value, $w_{i-1}(u)+w_{i-1}(v_j)<0$. Similarly, if $w_{i}(u)+w_{i}(v_j)\in \{1, 2\}$ then $w_{i-1}(u)+w_{i-1}(v_j)$ must be positive as well. 
\end{proof}

\begin{theorem} \label{thm:B A-code}
Let $F$ be a forest with $n$ vertices. The map \[\phi: \mathcal{W}_B(F) \rightarrow \{(w', (a_1, \dots,a_n)) : w' \in \mathcal{W}(F) \text{ is a natural labeling and } (a_1, \dots,a_n) \in \mathrm{SE}_F^B\}\] is a bijection.
\end{theorem}

\begin{proof} First note that by Lemma~\ref{phi} the map $\phi$ is well defined. We now describe the inverse of $\phi$. Given a pair $(w', (a_1,\dots,a_n))$ where $w' \in \mathcal{W}(F)$ is a natural labeling  and $(a_1, \dots,a_n) \in \mathrm{SE}_F^B$, the corresponding labelings $w_i$ from the definition of $\phi$ can be obtained in the following way.  First, $w_0=w'$. If $w_{i-1}$ has been constructed for $i \leq n$, let $A_{i} = \{|w_{i-1}(u)| : u \leq v_{i}\}$.  If $a_i<h_{v_i}$, find the $(a_i+1)$-st largest element in $A_i$, say it is $|w_{i-1}(v_j)|$, and set $w_i(v_i) =|w_{i-1}(v_j)|$.  If $h_{v_i}  \leq a_i <\leq 2h_{v_i} -1$ find the $(a_i - h_{v_i} +1)$-st smallest element of $A_i$, say it is $|w_{i-1}(v_j)|$, and set $w_i(v_i)=-|w_{i-1}(v_j)|$. In either case relabel the vertices below $v_i$ with the elements from $A_i \backslash \{|w_{i-1}(v_j)|\}|$ preserving the order of the original labels in absolute values as well as the signs at the vertices in $w_{i-1}$ (similarly to the third property in the definition of $\phi$ above), and call this new labeling $w_i$. The desired labeling $w$ that corresponds to $(w', (a_1,\dots,a_n))$ is simply $w_n$ constructed in this process. Note that similarly as in Lemma~\ref{phi} one can show that the $\acode$ of $(F, w_i)$ is $(a_1,\dots,a_i, 0, \dots, 0)$ and therefore the $\acode$ of $(F,w)$ is $(a_1, \dots, a_n)$.
\end{proof}

\begin{corollary} \label{cor:B same A-code}
Given a forest $F$ of size $n$ and a sequence $(a_1, \dots,a_n) \in \mathrm{SE}_F^B$, there are $\frac{n!}{\prod_{v \in V(F)}h_v}$ signed labelings of $F$ whose A-code is $(a_1, \dots,a_n)$.  
\end{corollary}

\begin{proof}
This follows from Theorem~\ref{thm:B A-code} and the well-known fact that there are $\frac{n!}{\prod_{v \in V(F)}h_v}$ natural labelings of the forest F. \end{proof}

\begin{theorem}
Let $F$ be a forest of size $n$. Then 
\begin{equation} \label{eq: invb,Btmaxb}
\sum_{w \in \mathcal{W}_B(F)} p^{\none(F,w)}q^{\invb(F,w)} \prod_{v \in \Btmaxb(F,w)} t_v = \frac {n!}{\prod_{v \in V(F)} h_{v} } \prod _{v \in V(F)} \left((1+pq^{h_v} )[h_v] -1 +t_v\right).
\end{equation}
\end{theorem}

\begin{proof}
This is a direct consequence of Lemma~\ref{lem:B A-code info}, Theorem~\ref{thm:B A-code}, and Corollary~\ref{cor:B same A-code}.
\end{proof}

As a corollary we obtain a generalization of the results of Bj\"orner and Wachs~\cite{BW1} and Chen, Gao, and Guo~\cite{CGG}.

\begin{corollary} \label{cor: invb,Btmaxb} Let $F$ be a forest of size $n$. Then 
\begin{equation}\label{inv unsigned} \sum_{w \in \mathcal{W}(F)} q^{\inv(F,w)} \prod_{v \in \Btmax(F,w)} t_v = \frac {n!}{\prod_{v \in V(F)} h_{v} } \prod _{v \in V(F)} \left([h_v] -1 +t_v\right)\end{equation}
\begin{equation} \label{inv signed} \sum_{w \in \mathcal{W}_B(F)} q^{\invb(F,w)} \prod_{v \in \Btmaxb(F,w)} t_v = \frac {n!}{\prod_{v \in V(F)} h_{v} } \prod _{v \in V(F)} \left([2h_v] -1 +t_v\right)\end{equation}
\begin{equation}\label{inv even} \sum_{w \in \mathcal{W}_D(F)} q^{\invd(F,w)} \prod_{v \in \Btmaxd(F,w)} t_v = \frac {n! \times 2^{\#\text{leaves in }F -1}}{\prod_{v \in V(F)} h_{v} }  \prod _{\substack{v \in V(F)\\\ v \text{ is not a leaf}}} \left((1+q^{h_v-1} )[h_v] -1 +t_v\right)\end{equation}
\end{corollary}
\begin{proof}
For $w \in \mathcal{W}_B(F)$ which is actually in $\mathcal{W}(F)$, $\inv(F,w) = \invb(F,w)$ and $\Btmax(F,w)= \Btmaxb(F,w)$. Therefore,~\eqref{inv unsigned} follows from~\eqref{eq: invb,Btmaxb} by setting $p=0$. The equation~\eqref{inv signed} is  obtained by setting $p=1$ in~\eqref{eq: invb,Btmaxb}. To get~\eqref{inv even}, let 
\[ D_n(p,q,t) =  \sum_{w \in \mathcal{W}_B(F)} p^{\none(F,w)}q^{\invd(F,w)} \prod_{v \in \Btmaxd(F,w)} t_v. \] Then
\begin{align*} D_n(p,q,t) &= \sum_{w \in \mathcal{W}_B(F)} p^{\none(F,w)}q^{\inv(F,w) +\ntwo(F,w)} \prod_{v \in \Btmaxb(F,w)} t_v \bigg|_{\substack{t_v = 1\\\ v \text{ is a leaf}}} \\
&= \sum_{w \in \mathcal{W}_B(F)} \left(\frac{p}{q}\right)^{\none(F,w)}q^{\invb(F,w)} \prod_{v \in \Btmaxb(F,w)} t_v \bigg|_{\substack{t_v = 1\\\ v \text{ is a leaf}}} \\
&\overset{\eqref{eq: invb,Btmaxb}}{=} \frac {n!}{\prod_{v \in V(F)} h_{v} } \prod _{v \in V(F)} \left((1+pq^{h_v-1} )[h_v] -1 +t_v\right) \bigg|_{\substack{t_v = 1\\\ v \text{ is a leaf}}} \\
&=  \frac {n!}{\prod_{v \in V(F)} h_{v} } (1+p)^{\#\text{leaves in }F} \prod _{\substack{v \in V(F)\\\ v \text{ is not a leaf}}} \left((1+pq^{h_v-1} )[h_v] -1 +t_v\right).
\end{align*}
Since $F$ has at least one leaf, $D_n(-1,q,t) = 0$, which implies 
\[\sum_{i \text{ is even}} [p^i] D_n(p,q,t)  = \sum_{i \text{ is odd}} [p^i] D_n(p,q,t), \] where $[p^i] D_n(p,q,t)$ denotes the coefficient in $D_n(p,q,t)$ in front of $[p^i] $. Therefore, 
\begin{align*} \sum_{w \in \mathcal{W}_D(F)} q^{\invd(F,w)} \prod_{v \in \Btmaxd(F,w)} t_v &= \sum_{i \text{ is even}} [p^i] D_n(p,q,t) \\
&= \frac{D_n(1,q,t)}{2} \\
&= \frac {n! \times 2^{\#\text{leaves in }F -1}}{\prod_{v \in V(F)} h_{v} }  \prod _{\substack{v \in V(F)\\\ v \text{ is not a leaf}}} \left((1+q^{h_v-1} )[h_v] -1 +t_v\right). \end{align*}

\end{proof}

Now we show how Corollary~\ref{cor: invb,Btmaxb} generalizes results for permutations.  Recall that for a permutation $\sigma \in S_n$, its length $\ell(\sigma)$ as an element in a Coxeter group with the standard generators is equal to the number of inversions, i.e., $\ell(\sigma) = \inv(\sigma) = \#\{(i,j): 1 \leq i <j \leq n, \sigma_i > \sigma_j\}$. For an element $\sigma \in B_n$, the length function is given by
\[ \ell_B(\sigma) = \inv(\sigma) + \none(\sigma) + \ntwo(\sigma),\]
where 
\[\none(\sigma)=\#\{i: 1\leq i \leq n, \sigma_i <0\}, \]
\[\ntwo(\sigma) = \#\{(i,j): 1 \leq i <j \leq n, \sigma_i + \sigma_j< 0\}.\] For an even-signed permutation $\sigma \in D_n$, the length function can be computed as
\[\ell_D(\sigma) = \inv(\sigma) + \ntwo(\sigma).\]
Moreover, for $\sigma \in S_n$, the set of right-to-left minimum letters is defined by
\[\Rlmin(\sigma) = \{\sigma_i: \sigma_i < \sigma_j \text{ for all } j >i \}.\] The type B right-to-left minimum letters for $\sigma \in B_n$ are defined by
\[\Rlminb(\sigma) = \{\sigma(i):  0<\sigma(i) < |\sigma_j| \text{ for all } j >i \}.\] The type D right-to-left minimum letters for $\sigma \in D_n$ are defined by
\[\Rlmind(\sigma) = \{\sigma(i): 1<\sigma(i) < |\sigma_j| \text{ for all } j >i \}.\]

Let $F$ be a tree of size $n$ with one leaf whose vertices are naturally indexed and $w \in \mathcal{W}_B(F)$. Let $\sigma$ be the signed permutation obtained by reading the labeling $w$ of $F$ from bottom to top, i.e., $\sigma = w(v_1) \cdots w(v_n)$. Then clearly $\inv(F,w) = \inv(\sigma)$, $\none(F,w) = \none(\sigma)$, $\ntwo(F,w) = \ntwo(\sigma)$. Therefore, if $\sigma \in S_n$, $\ell(\sigma) = \inv(F,w)$, if $\sigma \in B_n$, $\ell_B(\sigma) = \inv_B(F,w)$, and if $\sigma \in D_n$, $\ell_D(\sigma) = \inv_D(F,w)$.

The statistic $\Rlmin$ is related to $\Btmax$ in the following way.
\begin{lemma} \label{lem inv} Let $F$ be a linear tree with $n$ vertices. Let $w \in \mathcal{W}_B(F)$ and let $\sigma = w(v_1)\cdots w(v_n)$ be the corresponding signed permutation.  Then $\Btmaxb(F,w) = \Rlminb(\sigma^{-1})$. Moreover, if $w \in \mathcal{W}(F)$ then $\Btmax(F,w)= \Rlmin(\sigma^{-1})$ and if $w \in \mathcal{W}_D(F)$ then $\Btmaxd(F,w) = \Rlmind(\sigma^{-1})$.
\end{lemma}

\begin{proof}
Assume that the first statement holds for all linear trees of size at most $n$.  Let $F$ be a linear tree of size $n+1$ and let $w \in \mathcal{W}_B(F)$.  Now let $F'$ be the tree of size $n$ obtained by removing the root $v_{n+1}$ of $F$, and let $w' \in \mathcal{W}_B(F')$ be the corresponding standardized labeling obtained from $w$ by decreasing the absolute values of all labels larger than $|w(v_{n+1})|$ by 1 and preserving the signs.  Note that $v_{n+1} \in \Btmax_B(F)$  if and only if $w(v_{n+1}) = n+1$.  Therefore, if $w(v_{n+1)} \neq n+1$ then $\Btmaxb{(F,w)} = \Btmaxb{(P',w')}$, and if $w(v_{n+1}) = n+1$ then $\Btmaxb{(F,w)} = \Btmaxb{(F',w')} \cup \{v_{n+1}\}$.  Let $\sigma = w(v_1)  \cdots w(v_{n+1})$, and $\sigma'  = w'(v_1)  \cdots w'(v_n)$.  The permutation $\sigma'$ can be obtained from $\sigma$ by deleting the last letter $w(v_{n+1})$ and standardizing, and therefore, $\sigma'^{-1}$ is obtain by removing $n+1$ or $\overline{(n+1)}$ from $\sigma^{-1}$.  Since $w(v_{n+1})$ determines the position of $n+1$ in $\sigma^{-1}$, $n+1 \in \Rlminb(\sigma^{-1})$  if and only if $w(v_{n+1}) = n+1$.  Applying this and the induction hypothesis we have that if $w(v_{n+1}) \neq n+1$ then $\Rlminb{(\sigma^{-1})} = \Rlminb{(\sigma'^{-1})} = \Btmaxb{(F',w')} = \Btmaxb{(F,w)}$, and if $w(v_{n+1}) = n+1$ then $\Rlminb{(\sigma^{-1})} = \Rlminb{(\sigma'^{-1})} \cup \{n+1\} = \Btmaxb{(F',w')} \cup \{n+1\} = \Btmaxb{(F,w)}$.  Therefore $\Btmaxb{(F,w)} = \Rlminb{(\sigma^{-1})}$.

If there are no signed letters, the same argument shows that $\Btmax(F,w)= \Rlmin(\sigma^{-1})$. 
\end{proof}

Since $\inv(\sigma) = \inv(\sigma^{-1})$ and the corresponding statement is also true for signed and even signed permutations, as a direct consequence of Lemma~\ref{lem inv} and Corollary~\ref{cor: invb,Btmaxb} we get the following results for permutations.
\begin{corollary}
\begin{equation} \label{cor perm} \sum_{\sigma \in S_n} q^{\inv(\sigma)} \prod _{i \in \Rlmin(\sigma)} t_i = \prod _{i=1}^n ([i] -1 + t_i) \end{equation}
\begin{equation} \label{cor sign} \sum_{\sigma \in B_n} q^{\invb(\sigma)} \prod _{i \in \Rlminb(\sigma)} t_i = \prod _{i=1}^n \left([2i] -1 +t_i\right) \end{equation}
\begin{equation} \label{cor even} \sum_{\sigma \in D_n} q^{\invd(\sigma)} \prod _{i \in \Rlmind(\sigma)} t_i = \prod _{i=2}^n \left((1+q^{i-1} )[i] -1 +t_i\right) \end{equation}
\end{corollary}

Equation~\eqref{cor perm} was first shown in~\cite{BW2}, while~\eqref{cor sign} and~\eqref{cor even} can be found in~\cite{Poznanovic}, where a more general case of restricted permutations was also studied.

\section{Sorting Index  and Cycles}\label{S: sor}

In this section we introduce two new statistics for labeled forests, sorting index and cycle minima, and study their joint distribution. They are both motivated by corresponding permutation statistics which we first recall.
A signed permutation $\sigma \in B_n$ has a unique factorization as a product of signed transpositions, $\sigma = (i_1,j_1) \cdots (i_k,j_k)$, where $i_s<j_s$ for $1\leq s \leq k$ and  $0<j_1<\cdots<j_k$.  Then \[\sorb(\sigma) = \sum_{r=1}^k \left( j_r-i_r-\chi(i_r<0) \right).\]  Similarly as for unsigned permutations, the desired decomposition into transpositions can be  computed via application of the so called Straight Selection Sort algorithm which first places $n$ in the $n$-th position by applying a transposition, then places $n-1$ in the $(n-1)$-st position by applying a transposition, etc.
For example, for the signed permutation $\sigma = 4\bar{2}15\bar{3}$ we have \[4\bar{2}15\bar{3} \overset{(45)} \rightarrow 4\bar{2}1\bar{3}5 \overset{(14)} \rightarrow \bar{3}\bar{2}145 \overset{(\bar{1}3)} \rightarrow \bar{1}\bar{2}345 \overset{(\bar{2}2)} \rightarrow \bar{1}2345 \overset{(\bar{1}1)} \rightarrow 12345.\] Therefore, $\sorb(\sigma) = (1-(-1)-1)+(2-(-2)-1)+(3-(-1)-1)+(4-1)+(5-4) = 11$. It is not difficult to see that if $\sigma \in S_n$, then $\sor_B(\sigma) = \sor(\sigma)$ 

The set of minimal elements of the cycles of $\sigma \in S_n$ is denoted by $\Cyc(\sigma)$. Signed permutations can be decomposed into two types of cycles: balanced and unbalanced. The balanced cycles are of the form $(a_1, \dots, a_k)$ (this cycle also takes $\overline{a}_1$ to $\overline{a}_2$, etc.) while the unbalanced cycles are of the form $(a_1, \dots, a_k, \overline{a_1}, \dots, \overline{a_k})$, for $k \geq 1$ and all $a_1, \dots, a_k$ different.  For a signed permutation $\sigma \in B_n$ we let \[ \Cycb(\sigma) = \{|m| : m  \text{ is a minimal number in absolute value in a balanced cycle of } \sigma \}.\]

Now we introduce the sorting index for signed labeled forests. It is computed via a sorting algorithm related to Straight Selection Sort. To describe it we introduce the following notation. For a signed forest $(F, w)$, and a vertex $v$, $w_v$ will denote the labeling of the subtree of $F$ rooted at $v$ which is induced by $w$. 
The algorithm for computing the sorting index of type B, $\sorb(F,w)$, is as follows \\
\begin{lstlisting}[mathescape]
$\sorb(F,w)$=0
for i in range(n,1,-1) 
begin
	let $v$ be the vertex with $|w(v)|=i$ and let $u$ be the
	largest vertex such that $u \geq_F v$ and $|w(u)|\leq i$
	
	if $w(u)>0$ 
		$\sorb(F,w)$ = $\sorb(F,w)$ + $|w_u(v)|$ - $ w_u(u)$ 
	else 
		$\sorb(F,w)$ = $\sorb(F,w)$ + $|w_u(v)|$ - $w_u(u)$ - 1
	
	if $w(v)>0$
		interchange the labels on the vertices $u$ and $v$
	else
		multiply $w(u)$, and $w(v)$ by -1, and then 
		interchange the labels on the vertices $u$ and $v$		
		
	call the new labeling $w$	
end
\end{lstlisting}

For $w \in \mathcal{W}(F)$, we will define the type A sorting index, $\sor(F,w)$,  to be the same as $\sorb(F,w)$. Since in this case there are no negative labels, the sorting algorithm can be simplified and we present it here for convenience.
\begin{lstlisting}[mathescape]
$\sor(F,w)$=0
for i in range(n,1,-1) 
begin
	let $v$ be the vertex with $|w(v)|=i$ and let $u$ be the
	largest vertex such that $u \geq_F v$ and $w(u)\leq i$
	
	$\sor(F,w)$ = $\sor(F,w)$ + $w_u(v)$ - $w_u(u)$ 
	
	interchange the labels on the vertices $u$ and $v$ 
	and call the new labeling $w$
end
\end{lstlisting}
An example of sorting a signed labeled tree is given in Figure~\ref{fig: sorting}. For this tree we have : $\sor(F,w) = (5-2) + (1 +1-1) + (3-1) + (1+1-1) + (1-1) = 7$. 

The sorting algorithm applied to a labeling $w$ produces a positive natural labeling $w'$ of $F$. While for signed labelings $w$, the map $w \circ (w')^{-1}$ is technically a map $\{1, 2, \dots, n\} \rightarrow \{\pm 1, \pm 2, \dots, \pm n\}$, it can be uniquely extended to a signed permutation in $B_n$. 
\begin{definition} For $w \in \mathcal{W}(F)$, we define the \emph{minimal cycle vertices} of $(F,w)$ to be
\[\Cyc(F,w)=\{ v : w'(v) \in \Cyc(w \circ (w')^{-1})\}.\]
 For $w \in \mathcal{W}_B(F)$, we define the \emph{type B minimal cycle vertices} of $(F,w)$ to be
\[\Cycb(F,w)=\{ v : w'(v) \in \Cycb(w \circ (w')^{-1})  \}.\]
\end{definition}

For example, the signed permutation that corresponds to the signs labeled tree from Figure~\ref{sgnl} is $w \circ (w')^{-1} = 3\bar{5}1\bar{4}2 = (1 3) (2 \bar{5} \bar{2} 5) ( 4 \bar{4} )$. Therefore, $\Cycb(F,w) = \{v_1\}$.

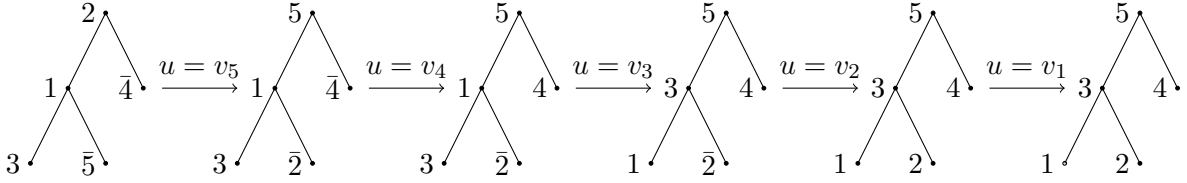
\begin{figure} 
\begin{center}
\begin{tikzpicture} [scale = 0.5]

\begin{scope}[xshift = 0.5cm]
\draw (0,0) -- (1,2) -- (2,4) -- (3,2);
\draw (1,2) -- (2,0);
\draw [fill] (0,0) circle [radius=0.05];
\node [left] at (0,0) {3};
\draw [fill] (1,2) circle [radius=0.05];
\node [left] at (1,2) {1};
\draw [fill] (2,4) circle [radius=0.05];
\node [left] at (2,4) {2};
\draw [fill] (3,2) circle [radius=0.05];
\node [left] at (3,2) {$\bar{4}$};
\draw [fill] (2,0) circle [radius=0.05];
\node [left] at (2,0) {$\bar{5}$};

\draw [->] (3.5,2) -- (5.5,2) node[midway,above] {$u=v_5$};
\end{scope}

\begin{scope}[xshift = 1cm]
\draw (5,0) -- (6,2) -- (7,4) -- (8,2);
\draw (6,2) -- (7,0);
\draw [fill] (5,0) circle [radius=0.05];
\node [left] at (5,0) {3};
\draw [fill] (6,2) circle [radius=0.05];
\node [left] at (6,2) {1};
\draw [fill] (7,4) circle [radius=0.05];
\node [left] at (7,4) {5};
\draw [fill] (8,2) circle [radius=0.05];
\node [left] at (8,2) {$\bar{4}$};
\draw [fill] (7,0) circle [radius=0.05];
\node [left] at (7,0) {$\bar{2}$};

\draw [->] (8.5,2) -- (10.5,2) node[midway,above] {$u=v_4$};
\end{scope}

\begin{scope}[xshift = 1.5cm]
\draw (10,0) -- (11,2) -- (12,4) -- (13,2);
\draw (11,2) -- (12,0);
\draw [fill] (10,0) circle [radius=0.05];
\node [left] at (10,0) {3};
\draw [fill] (11,2) circle [radius=0.05];
\node [left] at (11,2) {1};
\draw [fill] (12,4) circle [radius=0.05];
\node [left] at (12,4) {5};
\draw [fill] (13,2) circle [radius=0.05];
\node [left] at (13,2) {4};
\draw [fill] (12,0) circle [radius=0.05];
\node [left] at (12,0) {$\bar{2}$};

\draw [->] (13.5,2) -- (15.5,2) node[midway,above] {$u=v_3$};
\end{scope}

\begin{scope} [xshift = 2cm]
\draw (15,0) -- (16,2) -- (17,4) -- (18,2);
\draw (16,2) -- (17,0);
\draw [fill] (15,0) circle [radius=0.05];
\node [left] at (15,0) {1};
\draw [fill] (16,2) circle [radius=0.05];
\node [left] at (16,2) {3};
\draw [fill] (17,4) circle [radius=0.05];
\node [left] at (17,4) {5};
\draw [fill] (18,2) circle [radius=0.05];
\node [left] at (18,2) {4};
\draw [fill] (17,0) circle [radius=0.05];
\node [left] at (17,0) {$\bar{2}$};

\draw [->] (18.5,2) -- (20.5,2) node[midway,above] {$u=v_2$};
\end{scope}

\begin{scope} [xshift = 2.5cm]
\draw (20,0) -- (21,2) -- (22,4) -- (23,2);
\draw (21,2) -- (22,0);
\draw [fill] (20,0) circle [radius=0.05];
\node [left] at (20,0) {1};
\draw [fill] (21,2) circle [radius=0.05];
\node [left] at (21,2) {3};
\draw [fill] (22,4) circle [radius=0.05];
\node [left] at (22,4) {5};
\draw [fill] (23,2) circle [radius=0.05];
\node [left] at (23,2) {4};
\draw [fill] (22,0) circle [radius=0.05];
\node [left] at (22,0) {2};

\draw [->] (23.5,2) -- (25.5,2) node[midway,above] {$u=v_1$};
\end{scope}

\begin{scope} [xshift = 3cm]
\draw (25,0) -- (26,2) -- (27,4) -- (28,2);
\draw (26,2) -- (27,0);
\draw  (25,0) circle [radius=0.05];
\node [left] at (25,0) {1};
\draw [fill] (26,2) circle [radius=0.05];
\node [left] at (26,2) {3};
\draw [fill] (27,4) circle [radius=0.05];
\node [left] at (27,4) {5};
\draw [fill] (28,2) circle [radius=0.05];
\node [left] at (28,2) {4};
\draw [fill] (27,0) circle [radius=0.05];
\node [left] at (27,0) {2};
\end{scope}

\end{tikzpicture}
\caption {Sorting of the signed labeled tree from Figure~\ref{sgnl}. } 
\label{fig: sorting}
\end{center}
\end{figure}

Note that each vertex in $F$ plays the role of $u$ in the sorting algorithm exactly once. We define the $\bcode$ of $(F,w)$ to be the integer sequence $(b_1,\dots,b_n) \in \mathrm{SE}_F^B$ where $b_i$ is equal to the amount added to the sorting index in the step of the algorithm when $u = v_i$, (i.e. $b_i = |w_u(v)| - w_u(u) - 1$ or $b_i = |w_u(v)| -  w_u(u)$). One can think of $b_i$ as  the amount contributed to the sorting index by the vertex $v_i$.  For example, for the tree in Figure~\ref{sgnl}, $\bcode(F,w) = (0, 1, 2, 1, 3)$.

\begin{lemma} \label{lem:B B-code info}
Let $w \in \mathcal{W}_B(F)$ and let  $\bcode(F,w) = (b_1,b_2,\dots,b_n)$. Then 
\begin{enumerate}
\item $\sorb(F,w) = \sum _{i=1}^n b_i$ 
\item $\Cycb(F,w) = \{v_i : b_i = 0\}$
\item $w \in \mathcal{W}(F)$ if and only if $b_i < h_{v_i}$ for all $i$.
\end{enumerate}
\end{lemma}

\begin{proof}
The first part follows from the way $b_i$ is defined. 

For the second part we will use induction on $n$, the size of $F$.  Suppose that the statement is true for all forests of size less than $n$. First assume $F$ is a forest with trees $T_1,T_2,\dots,T_k$ for some $k>1$.  The $\bcode$ of $F$ is a concatenation of the B-codes of $(T_1, w_1),(T_2, w_2),\dots,(T_k, w_k)$ (with possible rearrangements depending on the indexing of the vertices), where $w_j$ is $w$ restricted to $T_j$.  The vertex $ v_i$  is in $\Cycb(F,w)$ if and only if for some $j \leq k$, $v_i \in \Cycb(T_j, w_j)$ .  By the induction hypothesis $v_i \in V(T_j)$ is an element of  $\Cycb(T_j,w_j)$ if and only if $b_i = 0$.  Therefore $i$ is an element of $\Cycb(F,w)$ if and only if $b_i= 0$.

Now assume that $k=1$, i.e, $F$ is a tree.  Let $F_1$ be the forest obtained by removing the root $v_n$ from the tree $F$, and let $w_1$ be the labeling obtained by restricting $w$ to $F_1$ and replacing the label $n$ with $w(v_n)$. Now let $w'$ and $w_1'$ be the labelings of $F$ and $F_1$, respectively, obtained by sorting $w$ and $w_1$, respectively.  The permutation $w_1 \circ w_1'^{-1}$ can be obtained from $w \circ w'^{-1}$ by deleting the elements $n$ and $\overline{n}$ in the cycle notation of $w \circ w'^{-1}$.  Thus for $i=1,\dots, n-1$, $v_i \in \Cycb(F_1,w_1)$ if and only if $v_i \in \Cyc_B(F,w)$.  Applying the induction hypothesis, for all $i=1,\dots, n-1$, $v_i \in \Cycb(F,w)$ if and only if $b_i=0$.  The value $n$ is a minimal element of a balanced cycle in $w \circ w'^{-1}$ if and only if it is in a cycle by itself and thus $w(v_n) = n$, which happens exactly when $b_n = 0$.  Therefore, $\Cycb(F,w) = \{ v_i : b_i = 0\}$.

For the third part, note that the value $|w_u(v)|$ that appears in the sorting algorithm is equal to $h_u$. Therefore, the contribution of $u$ to $\sorb(F,w)$ is less than $h_u$ if and only if the current label of the vertex $u$ is positive. Because of the rule of interchanging the signs of the labels during the sorting process, if the starting labeling $w$ has at least one negative sign there will be a step in the process in which $u$ has a negative label. On the other hand, if $w \in \mathcal{W}(F)$, then all the labels remain positive throughout the sorting.
\end{proof}

Similarly to the $\acode$, the $\bcode$ also induces a map $\psi$ from $\mathcal{W}_B(F)$ onto the set of pairs $(w', (b_1, \dots, b_n))$ of a natural positive labeling $w'$ of $F$ and a sequence $(b_1, \dots, b_n) \in \mathrm{SE}_F^B$. The natural labeling $w'$ is the one obtained by sorting $w$, while  $(b_1, \dots, b_n)$ is the B-code of $(F,w)$.

\begin{theorem} \label{thm sor} Let $F$ be a forest with $n$ naturally indexed vertices $v_1, \dots, v_n$. The map \[ \psi : \mathcal{W}_B(F) \rightarrow \{(w', (b_1, \dots,b_n)) : w' \in \mathcal{W}(F) \text{ is a natural labeling  and }(b_1, \dots,b_n) \in \mathrm{SE}_F^B\} \] is a bijection. Restricted on positive labelings, $\psi$ is a bijection from $\mathcal{W}(F)$ to the set of pairs $(w', (b_1, \dots,b_n))$ where $w' \in \mathcal{W}(F)$ is a natural labeling  and $(b_1, \dots,b_n) \in \mathrm{SE}_F$.
\end{theorem}

\begin{proof} We describe the inverse of $\psi$. Given a pair $(w',(b_1,\dots, b_n))$  of  a natural labeling $w' \in \mathcal{W}(F)$ and $(b_1, \dots,b_n) \in \mathrm{SE}_F^B$, the original labeling $w$ can be recovered in the following way.  Begin with $i=1$, and let $j$ be such that $|w'(v_j)| =i$. Let $B_i = \{|w'(v)| : v\leq_F v_j\}$.  If $b_j < h_{v_j}$, let $u$ be the vertex so that $|w'(u)|$ is the $(b_i+1)$-st largest element in $B_i$. Then if $w'(u) > 0$ simply interchange the labels of $u$ and $v_j$.  Otherwise, first change the signs of the labels of $u$ and $v_j$ and then interchange them. Otherwise, if $h_{v_j} \leq b_j < 2h_{v_j}$, let $u$ be the vertex so that $|w'(u)|$ is the $(b_j - h_{v_j} +1)$-st smallest element of $B_i$. Then if $w'(u) < 0$ simply interchange the labels of $u$ and $v_j$.  Otherwise, first change the signs of the labels of $u$ and $v_j$ and then interchange them.  Keep calling the new labeling $w'$. Repeat for $i=2,\dots,n$.  The final labeling is the desired $w \in \mathcal{W}_B(F)$.

The second part of the theorem follows from the third part of Lemma~\ref{lem:B B-code info}.
\end{proof}

\begin{corollary} \label{cor: sorB,CycB}
Let $F$ be a forest of size $n$, then 
\begin{equation} \label{eq: sor,Cyc}
\sum_{w \in \mathcal{W}(F)} q^{\sor(F,w)} \prod_{v \in \Cyc(F,w)} t_v = \frac {n!}{\prod_{v \in V(F)}h_v} \prod _{v \in V(F)} \left([h_{v}] - 1 + t_{v}\right),
\end{equation}
\begin{equation} \label{eq: sorb,Cycb}
\sum_{w \in \mathcal{W}_B(F)} q^{\sorb{(F,w)}} \prod_{v \in \Cycb{(F,w)}} t_v = \frac {n!}{\prod_{v \in V(F)}h_v } \prod _{v \in V(F)} \left([2h_{v}] - 1 + t_v\right).
\end{equation}
\end{corollary}

\begin{proof} This is a direct consequence of Lemma~\ref{lem:B B-code info} and Theorem~\ref{thm sor}. The products on the right-hand side of~\eqref{eq: sor,Cyc} and~\eqref{eq: sorb,Cycb} are the generating functions for the sequences in $\mathrm{SE}_F$ and $\mathrm{SE}_F^B$, respectively, according to total sum of elements and positions of zeros. The factor $n!/\prod_{v \in V(F)}h_v$ is due to the fact that the B-code is a $(n!/\prod_{v \in V(F)}h_v)$-to-1 map. 
\end{proof}

Our definition of $\sor$ and $\Cyc$ for labeled forests was motivated by corresponding permutation statistics. Petersen~\cite{Petersen} showed that
\[ \sum_{\sigma \in B_n} q^{\sorb(\sigma)} t^{\sorb(\sigma)} = \sum_{\sigma \in B_n} q^{\invb(\sigma)} t^{\rmilb(\sigma)}.\] This equidistribution was later generalized to include  $r$-colored permutations and, instead of just $\sor$ and $\cyc$, the result was refined in terms of set-valued statistics $\Rmil$ and $\Cyc$ as well  as additional statistics that allow to deduce results for restricted permutations~\cite{Poznanovic, CGG2, ELW}. 

The following two theorems reveal the relation of the statistics $\sor$ and $\Cyc$ for labeled forests with the corresponding permutation statistics.

\begin{theorem} \label{thm: sorb}
Let $F$ be a linear tree of size $n$ and $w \in \mathcal{W}_B(F)$.  Let $\sigma = w(v_1)  \cdots w(v_n)$ be the corresponding signed permutation, then $\sorb(F,w)= \sorb(\sigma^{-1})$. Consequently, if $w \in \mathcal{W}(F)$, then $\sor(F,w)= \sor(\sigma^{-1})$.
\end{theorem}

\begin{proof}
Assume that the theorem holds for all linear trees of size at most $n$.  Let $F$ be a linear tree of size $n+1$ and let $F_1$ be the tree of size $n$ obtained by removing the root $v_{n+1}$ of $F$.

Consider first the case when $n+1$ appears as a label in $w$. Let $w_1$ be the labeling of $F_1$ obtained by restricting $w$ to $F_1$ and replacing the label $n+1$ with $w(v_{n+1})$. If $w(v_{n+1})<0$ then $\sorb(F,w) = \sorb(F_1,w_1) + (n+1) - w(v_{n+1})  - 1$, and if $w(v_{n+1})>0$ then $\sorb(F,w) = \sorb(F_1,w_1) + (n+1) - w(v_{n+1})$. Now let $\sigma = w(v_1)  \cdots w(v_{n+1})$, and $\sigma_1 = w_1(v_1)  \cdots w_1(v_n)$. Note that $\sigma_1^{-1}$ is the permutation obtained from $\sigma^{-1}$ by performing the first step of the Straight Selection Sort Algorithm and deleting $n+1$. Moreover, $w(v_{n+1})$ is the position of $n+1$ in $\sigma^{-1}$. Therefore, if $w(v_{n+1}) < 0$ then $\sorb(\sigma^{-1}) = \sorb(\sigma_1^{-1}) + (n+1) - w(v_{n+1})  - 1$, and if $w(v_{n+1}) >0$ then $\sorb(\sigma^{-1}) = \sorb(\sigma_1^{-1}) + (n+1) - w(v_{n+1})$. The claim follows by applying the induction hypothesis.

The proof in the case  when $\overline{n+1}$ appears as a label in $w$ is similar. We set $w_1$ to be the labeling of $F_1$ obtained by restricting $w$ to $F_1$ and replacing the label $\overline{n+1}$ with $\overline{w(v_{n+1})}$. If $w(v_{n+1})<0$ then $\sorb(F,w) = \sorb(F_1,w_1) + (n+1) - w(v_{n+1})  - 1$, and if $w(v_{n+1})>0$ then $\sorb(F,w) = \sorb(F_1,w_1) + (n+1) - w(v_{n+1})$. So, we can continue as above. 

Finally, note  that if a labeling or a permutation are in fact positive then the type B and type A sorting index coincide.
\end{proof}

\begin{theorem} \label{thm: Cycb}
Let $F$ be a linear tree with naturally indexed vertices $v_1, \dots, v_n$.  Let $w \in \mathcal{W}_B(F)$ and let  $\sigma = w(v_1) w(v_2) \cdots w(v_n) \in B_n$. Then $ v_i \in \Cycb{(F,w)}$ if and only if  $i \in \Cycb(\sigma^{-1})$. Consequently, if $w \in \mathcal{W}(F)$, then $v_i \in \Cyc(F,w)$ if and only if $i \in \Cyc(\sigma^{-1})$.
\end{theorem}

\begin{proof}
By definition, $\Cycb{(F,w)} = \{v_i :  w'(v_i) \in \Cycb(w \circ w'^{-1})\}$. For a linear tree  the sorted labeling $w'$ is  given by $w'(v_i) = i$.  Therefore for every $w \in \mathcal{W}_B(P)$, $w \circ w'^{-1} = \sigma$, and hence $v_i \in \Cycb(F,w)$ if and only if  $i \in \Cycb(\sigma)$. The claim then follows from the fact that $\Cycb(\sigma)=\Cycb(\sigma^{-1})$.
\end{proof}

As a corollary to Corollary \ref{cor: sorB,CycB}, Theorem \ref{thm: sorb}, and Theorem \ref{thm: Cycb} we get the following result.

\begin{corollary} [\cite{CGG2, Poznanovic}]
\[ \sum_{\sigma \in S_n} q^{\sor{(\sigma)}} \prod _{i \in \Cyc{(\sigma)}} t_i = \prod _{i=1}^n ([i] -1 + t_i)\]
\[\sum_{\sigma \in B_n} q^{\sorb{(\sigma)}} \prod _{i \in \Cycb{(\sigma)}} t_i = \prod _{i=1}^n ([2i] -1 + t_i). \]

\end{corollary}


\section{Major Index and Cyclic Bottom-to-Top Maxima}\label{S: maj}

While for permutations it is  true that 
\[ \sum_{\sigma \in S_n} q^{\inv(\sigma)} t^{\rmil(\sigma)} = \sum_{\sigma \in S_n} q^{\maj(\sigma)} t^{\rmil(\sigma)}, \]
for a  general forest $F$, $(\inv, \#\Btmax)$ and $(\maj,\# \Btmax)$ are not equidistributed over $\mathcal{W}(F)$. In this section we find a suitable Stirling partner for $\maj$ for labeled forests and then discuss the case of signed labelings. 

\begin{definition}
Let $(F,w)$ be a labeled forest. A vertex $v$ is a \emph{cyclic bottom-to-top maximum} if its label is a bottom-to-top maximum with respect to the cyclic shift of the natural ordering of the integers $1, 2, \dots, n$ beginning with the label of the parent of $v$, $p$. Precisely, if $w(v) < w(p)$, then $v$ is a cyclic bottom-to-top maximum if 
\[ \{u: u <_F v, w(u) \in [w(v), w(p)] \} = \emptyset.\] If $w(p) < w(v)$, then $v$ is a cyclic bottom-to-top maximum if 
\[ \{u: u <_F v, w(u) \notin [w(p), w(v)] \} = \emptyset.\] Let $\Cbtmax{F,w}$ denote the set of all cyclic bottom-to-top maxima of the labeled forest $(F,w)$.
\end{definition}

Let $F$ be a forest of size $n$ with naturally indexed vertices $\{v_1, v_2, \ldots, v_n\}$.  We will denote by $p_i$ be the parent of $v_i$, and while for a root $v_j$, $p_j$ is not defined, we will  use the convention $w(p_j)=n+1$.  Define the $\mcode$ $(s_1,s_2, \ldots, s_n)$ of $(F,w)$ as follows (see Figure~\ref{fig: mcode}):
\begin{align*}
& m_i = \# \{u : u <_F v_i, w(u) \in [w(v_i), w(p_i)]\}  \text{ if } w(v_i)<w(p_i) \\
& m_i = \# \{u :  u <_F v_i, w(u) \notin [w(p_i), w(v_i)]\} \text{ otherwise}.
\end{align*}
The special case of this code for permutations was used under the name ``McMahon code'' in~\cite{Vajnovszki}, where its relationship to the Lehmer code was explained. 

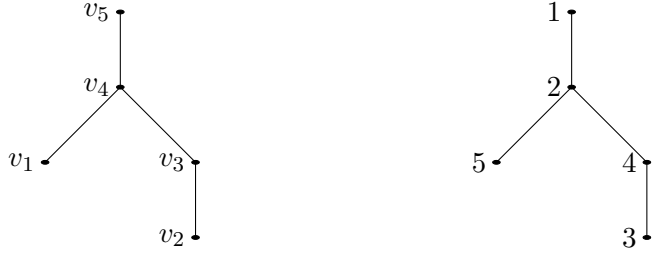
\begin{figure}
\begin{center}
\begin{tikzpicture} [yscale = 0.5]
\begin{scope}
\draw (0,0)--(1,2)--(1,4);
\draw (1,2)--(2,0)--(2,-2);
\draw [fill] (0,0) circle [ radius = 0.05];
\node [left] at (0,0) {$v_1$};
\draw [fill] (1,2) circle [ radius = 0.05];
\node [left] at (1,2) {$v_4$};
\draw [fill] (1,4) circle [ radius = 0.05];
\node [left] at (1,4) {$v_5$};
\draw [fill] (2,0) circle [ radius = 0.05];
\node [left] at (2,0) {$v_3$};
\draw [fill] (2,-2) circle [ radius = 0.05];
\node [left] at (2,-2) {$v_2$};
\end{scope}

\begin{scope} [xshift = 6cm]
\draw (0,0)--(1,2)--(1,4);
\draw (1,2)--(2,0)--(2,-2);
\draw [fill] (0,0) circle [ radius = 0.05];
\node [left] at (0,0) {5};
\draw [fill] (1,2) circle [ radius = 0.05];
\node [left] at (1,2) {2};
\draw [fill] (1,4) circle [ radius = 0.05];
\node [left] at (1,4) {1};
\draw [fill] (2,0) circle [ radius = 0.05];
\node [left] at (2,0) {4};
\draw [fill] (2,-2) circle [ radius = 0.05];
\node [left] at (2,-2) {3};
\end{scope}
\end{tikzpicture}
\caption{For this labeled tree, we have $\Cyc(F,w) = \{v_1, v_2, v_3\}$ and $\mcode(F,w) = (0,0,0,3,4)$. Also, $\Des(F,w) = \{v_1, v_3, v_4\}$ and therefore $\maj(F,w) = 1 + 2 + 4 =7$.}
\label{fig: mcode}
\end{center}
\end{figure}

\begin{theorem} \label{maj}
Let $(m_1, m_2, \dots, m_n)$ be the $\mcode$ of a labeled forest $(F,w)$. Then $\sum _{i=1}^{n} m_i = \maj{(F,w)}$, and $m_i = 0$ if and only if $v_i \in \Cbtmax{(F,w)}$.
\end{theorem}

\begin{proof}
The second part follows directly from the definitions. 

For the first part, assume  that the statement holds for all forests of size less than $n$. Suppose first that $F$ is a forest of size $n$ with trees $T_1, T_2, \ldots, T_k$ for $k>1$.   It is clear that $\maj{(P,w)} = \sum _{i=1}^{n} \maj{(T_i,w_i)} $, where $w_i$ is the labeling of $T_i$ induced by $w$.  Also, the $\mcode$ $(m_1, m_2, \dots, m_n)$ of $(F,w)$ is just a concatenation of the M-codes of the labeled trees $(T_1, w_1),(T_2, w_2),\dots,(T_k,w_k)$, with reordering as necessary.  Therefore $\sum _{i=1}^{n} m_i = \maj{(F,w)}$. 

Now suppose $k=1$.  Note that $v_{n-1}$ is a child of $v_n$. Let $F'$ be the forest obtained by deleting the edge $(v_{n-1}, v_n)$ from $F$. Denote by $(m'_1, m'_2, \dots, m'_{n})$  the $\mcode$ fof $(F', w)$. Let $A = \{u : u <_F v_{n-1} \text{ and } w(u)<w(v_{n-1})\}$, $B= \{u : u <_F v_{n-1} \text{ and } w(v_{n-1}) <  w(u) < w(v_n) \}$, and $C = \{u  : u <_F v_{n-1} \text{ and } w(v_n)<w(u)\}$.  We will consider two cases.\\
\textbf{Case 1.} $w(v_{n-1})<w(v_n)$\\ 
In this case $\Des(F,w) = \Des(F', w)$ and hence $\maj(F,w) = \maj(F', w)$. Comparing the two M-codes,  we have  $m'_{n-1} = m_{n-1} + \#C$, $m'_n = m_n - \#C$, and $m'_i = m_i$ for all $i \neq n-1,n$. Therefore,
\[ \sum_{i=1}^n m_i= \sum_{i=1}^n m'_i = \maj(F', w) = \maj(F, w) .\]
\textbf{Case 2.} $w(v_{n-1})>w(v_n)$\\ 
In this case $\Des(F,w) = \Des(F', w) \cup \{v_{n-1}\}$ and hence $\maj(F,w) = \maj(F', w) + h_{v_{n-1}}$. Comparing the two MacMahon codes,  we have $m'_{n-1} =h_{v_{n-1}} - 1- \#A$, $m'_n = m_n - 1- \#C$, and $m'_i = m_i$ for all $i \neq n-1,n$. For the code of $(F,w)$, we notice that $m_{n-1} = \#\{u : u <_F v_{n-1} \text{ and } w(u)> w(v_{n-1})\} + \# \{u : u <_F v_{n-1} \text{ and } w(u)<w(v_{n})\} = (h_{v_{n-1}} - 1 - \#A) + (h_{v_{n-1}} - 1- \#C)$. Therefore, 
\[ \sum_{i=1}^n m_i= \sum_{i=1}^n m'_i + h_{v_{n-1}} = \maj(F', w) + h_{v_{n-1}} = \maj(F, w) .\]
\end{proof}

The $\mcode$ induces a map $\theta$ from $\mathcal{W}(F)$ to the set of pairs $(w', (m_1, \dots, m_n))$, where $w'$ is a natural labeling of $F$ and $(m_1, \dots, m_n) \in \mathrm{SE}_F$ defined in the following way. For $w \in \mathcal{W}(F)$, the corresponding subexcedent sequence $(m_1, \dots, m_n)$ is its $\mcode$. The natural labeling $w'$ is obtained in $n$ steps by sorting $w$ as follows. Start with $i=n$. Let $\ell_1 < \cdots \ell_{h_{v_i}}$ be the labels of the subtree rooted at $v_i$. Replace each label $\ell_j$ by $\ell_{j+m_i}$, where the addition is performed modulo $h_{v_i}$. Note that after this, the vertex $v_i$ is a cyclic bottom-to-top maximum in the new labeling, while the other cyclic bottom-to-top maxima remain unchanged. It is not difficult to see that the $\mcode$ of the new labeling is $(m_1, \dots, m_{i-1}, 0 \dots, 0)$. Decrease $i$ by 1 and repeat until $i=0$. This will produce a labeling $w'$ with $\mcode$ $(0, \dots, 0)$ so it is natural. Because of this discussion, it is not difficult to see that the steps are reversible and therefore $\theta$ is a bijection. We summarize this in the following theorem.

\begin{theorem}\label{theta map}  The map \[\theta : \mathcal{W}(F)  \rightarrow \{(w', (m_1, \dots, m_n)) : w' \in \mathcal{W}(F) \text{ is a natural labeling } \text{ and } (m_1, \dots, m_n) \in \mathrm{SE}_F \} \] described above is a bijection.
\end{theorem} 

As a corollary to Theorem~\ref{theta map} and Theorem~\ref{maj} we get the following result.

\begin{theorem}
 \begin{equation}
\sum_{w \in \mathcal{W}(F)} q^{\maj(P,w)} \prod_{v \in \Cbtmax(F,w)} t_v = \frac {n!}{\prod_{v \in V(F)} h_{v} } \prod _{v \in V(F)} \left([h_{v}] - 1 + t_{v}\right),
\end{equation}
\end{theorem}

Liang and Wachs~\cite{LW}, constructed a bijection on labeled forests to prove that the enumerator for the inversion index on labeled forests is identical to the enumerator for the major index on labeled forests. For the symmetric group their bijection reduces to a map similar to Foata's second fundamental transformation. Note that as a consequence of the properties of the $\acode$ and $\bcode$ for labeled forests, the map $\theta^{-1}\circ \phi: \mathcal{W}(F) \rightarrow \mathcal{W}(F)$ has the stronger property: it takes $(\inv, \Btmax)$ to $(\maj, \Cbtmax)$. This map is different from the one in~\cite{LW}.

In~\cite{CGG}, Chen, Gao, and Guo defined two major indices for signed labeled forests and showed that they are equidistributed with $\invb$. The first one is based on the flag major index for signed permutations introduced by Adin and Roichman~\cite{AR}. The second one is based on a mahonian statistic for signed permutations that implicitly appears in~\cite{Reiner}. 

\begin{definition}[\cite{CGG}]
For a signed labeled forest $(F, w)$, 
\[ \fmaj(F,w) = 2 \maj(F,w) + n_1(F,w).\]
\end{definition} 

For a signed forest $(F,w)$ let \[ \Desb(F,w) = \Des(F,w) \cup \{ u \in F : u \text{ is a root of } F \text{ with a positive label} \}\] and \[\majb{(F,w)} = \sum _{u \in \Desb(F,w)} h_u.\]  Let $\p(F,w)$ be the number of positive labels of $w$. 

\begin{definition} [\cite{CGG}]
For a signed labeled forest $(F,w)$, \[\rmaj(F,w) = 2\majb(F,w) - \p(F,w).\]
\end{definition}

As observed in~\cite{CGG}, there is a simple map that sends $\fmaj$ to $\rmaj$, so here we will discuss only finding a Stirling partner for $\fmaj$. One could define a generalization of $\Cbtmax$ for signed labelings as follows.

\begin{definition}
Let $(F,w)$ be a signed labeled forest of size $n$. A vertex $v$ is a \emph{cyclic bottom-to-top maximum} if its label is positive and is a bottom-to-top maximum with respect to the cyclic shift of the natural ordering of the integers $-n, \dots, -1, 1, \dots, n$  beginning with the label of the parent of $v$, $p$. Precisely, for a vertex $v$ with a positive label, if $w(v) < w(p)$, then $v$ is a cyclic bottom-to-top maximum if 
\[ \{u: u <_F v, w(u) \in [w(v), w(p)] \} = \emptyset.\] If $w(p) < w(v)$, then $v$ is a cyclic bottom-to-top maximum if 
\[ \{u: u <_F v, w(u) \notin [w(p), w(u)] \} = \emptyset.\] Let $\Cbtmaxb{F,w}$ denote the set of all cyclic bottom-to-top maxima of the signed labeled forest $(F,w)$.
\end{definition}

Unfortunately, the pairs $(\fmaj, \#\Cbtmax)$ and $(\invb, \#\Btmax)$ are not equidistributed over $\mathcal{W}_B(F)$.   It would be interesting to see if there is a better definition of $\Cbtmax(F,w)$ or if there is another natural Stirling partner for $\fmaj$. Here we will only show that there is an analog of Theorem~\ref{maj} for signed forests. 
 
For a signed labeled forest $(F,w)$ with naturally indexed vertices  $\{v_1, v_2, \ldots, v_n\}$ we define its signed $\mcode$ to be the sequence $(m_1, \dots, m_n)$ given by 
\begin{align*}
&m_i = 2 \#\{u :  u <_F v_i, w(u) \in [w(v_i), w(p_i)]\} + \chi (w(v_i)<0)  \;\;\; \text{ if } w(v_i)<w(p_i)\\
&m_i = 2 \# \{u :  u <_F v_i, w(u) \notin [w(p_i), w(v_i)]\} + \chi (w(v_i)<0)  \;\;\; \text{ otherwise}.
\end{align*}
Here we use the same convention that $p_i$ is the parent of $v_i$ and if $v_i$ is a root of $F$, then $w(p_i)=n+1$. 

\begin{theorem} \label{thm: fmaj}
For a forest $F$ with signed labeling $w$ and $\mcode$ $(m_1, m_2, \dots, m_n)$, $\sum _{i=1}^{n} m_i = \fmaj{(P,w)}$, and $m_i = 0$ if and only if $v_i \in \Cbtmaxb(F,w)$.
\end{theorem}

\begin{proof}
It is clear from the definitions that $m_i = 0$ if and only if $v_i \in \Cbtmaxb{(F,w)}$.

For the first part, we use induction on $n$, the number of vertices of $F$. If $F$ is a forest, the claim follows the same way as in the unsigned case (Theorem~\ref{maj}). Therefore, suppose that $F$ is a tree. Then $v_n$ is the root of $F$. Consider the child of $v_{n}$, $v_{n-1}$, and let $F'$ be the forest obtained from deleting the edge $(v_{n-1}, v_n)$ from $F$. Let $(m'_1, \dots, m'_n)$ be the signed MacMahon code of $F'$. We will use the sets $A = \{u \in F : u <_F v_{n-1} \text{ and } w(u)<w(v_{n-1})\}$, $B= \{u \in F : u <_F v_{n-1} \text{ and } w(v_{n-1}) < w(u) < w(v_n) \}$, and $C = \{u \in F : u <_F v_{n-1} \text{ and } w(v_n)<w(u)\}$. \\
\textbf{Case 1.} $w(v_{n-1})< w(v_n)$\\
In this case, $v_{n-1} \notin \Des(F,w)$ and therefore $\maj(F,w) = \maj(F',w)$. So, $\fmaj(F,w) = \fmaj(F',w)$.  Note that $m'_{n-1} = m_{n-1} + 2\#C$, $m'_n = m_n - 2\#C$, and $m'_j = m_j$ for all $j \neq n-1, n$.  Thus 
\[ \sum_{i=1}^n m_i= \sum_{i=1}^n m'_i  = \maj(F', w)  = \maj(F, w) .\]
\textbf{Case 2.} $w(v_{n-1})> w(v_n)$\\
In this case, $ \Des(F,w) = \Des(F',w) \cup \{v_{n-1}\}$ and therefore $\maj(F,w) = \maj(F',w) + h_{v_{n-1}}$. This implies $\fmaj(F,w) = \fmaj(F',w) + 2h_{v_{n-1}}$. 
Note that $m'_{n-1} = m_{n-1} - 2(h_{v_{n-1}} - 1- \#C)$, $m'_n = m_n - 2- 2\#C$, and $m'_j = m_j$ for all $j \neq n-1, n$.  Thus 
\[ \sum_{i=1}^n m_i= \sum_{i=1}^n m'_i  + 2h_{v_{n-1}}= \maj(F', w) + 2h_{v_{n-1}} = \maj(F, w) .\]
\end{proof}

The difference between the signed and the unsigned case is that the map from $\mathcal{W}_B(F) \rightarrow \mathrm{SE}_F^B$ given by the $\mcode$ is not onto.

\end{document}